\documentclass[11pt]{amsart}
\usepackage{amssymb}
\usepackage{amsmath}
\usepackage{setspace}
\usepackage{amsthm}
\usepackage{epsfig,color,colordvi,wasysym}
\usepackage{graphicx}
\usepackage{amsfonts}
\usepackage{transparent}

\newtheorem{theorem}{Theorem}

\newtheorem{corollary}[theorem]{Corollary}

\newtheorem{comment}[theorem]{Comment}

\newcommand\commentout[1]{}

\newcommand\Def[1]{\emph{#1}}

\newcommand{\lcm}{\operatorname{lcm}}

\def\l{\lambda}
\def\0{\mathbf{0}}
\def\1{\mathbf{1}}

\newcommand\qbinom[2]{\left[ {{#1} \atop {#2}} \right]}

\begin{document}

\title{Partitions with fixed differences between largest and smallest parts}

\author{George E. Andrews}
\address{Department of Mathematics, The Pennsylvania State University, University Park, PA 16802, USA}
\email{andrews@math.psu.edu}
\author{Matthias Beck}
\author{Neville Robbins}
\address{Department of Mathematics, San Francisco State University, San Francisco, CA 94132, USA}
\email{[mattbeck,nrobbins]@sfsu.edu}



\date{19 March 2015}

\thanks{We thank an anonymous referee for numerous helpful suggestions.
M.\ Beck's research was partially supported by the US National Science Foundation (DMS-1162638).}

\maketitle

\begin{abstract}
We study the number $p(n,t)$ of partitions of $n$ with difference $t$ between largest and smallest parts.
Our main result is an explicit formula for the generating function $P_t(q) := \sum_{ n \ge 1 } p(n,t) \, q^n$.
Somewhat surprisingly, $P_t(q)$ is a rational function for $t>1$; equivalently, $p(n,t)$ is a quasipolynomial in $n$ for fixed $t>1$.
Our result generalizes to partitions with an arbitrary number of specified distances.
\end{abstract}


\vspace{12pt}

Enumeration results on integer partitions form a classic body of mathematics going back to at least Euler, including numerous applications throughout mathematics and some areas of physics; see, e.g., \cite{andrewstheoryofpartitions}.
A \Def{partition} of a positive integer $n$ is, as usual, an integer $k$-tuple $\l_1 \ge \l_2 \ge \dots \ge \l_k > 0$, for some $k$, such that
\[
  n = \l_1 + \l_2 + \dots + \l_k \, .
\] 
The integers $\l_1, \l_2, \dots, \l_k$ are the \Def{parts} of the partition.
We are interested in the counting function
\[
  p(n,t) := \# \text{partitions of $n$ with difference $t$ between largest and smallest parts} .
\]
It is immediate that
\[
  p(n,0) = d(n)
\]
where $d(n)$ denotes the number of divisors of $n$. 
Charmingly, $p(n,1)$ equals the number of \emph{non}divisors of $n$:
\[
  p(n,1) = n - d(n) \, ,
\]
which can be explained bijectively by the fact that the partitions counted by $p(n,0) + p(n,1)$
contain exactly one sample with $k$ parts, for each $k = 1, 2, \dots, n$ \cite[Sequence A049820]{sloaneonlineseq}, or by the generating function identity
\[
  \sum_{ n \ge 1 } p(n,1) \, q^n
  = \sum_{ m \ge 1 } \frac{ q^m }{ 1-q^m } \, \frac{ q^{ m+1 } }{ 1-q^{ m+1 } } 
  = \frac{ q }{ (1-q)^2 } - \sum_{ m \ge 1 } \frac{ q^m }{ 1-q^m } \, .
\]
(The last equation follows from a few elementary operations on rational functions).
An even less obvious instance of our partition counting function is
\begin{equation}\label{eq:caset2}
  p(n,2) = \binom{ \left\lfloor \frac n 2 \right\rfloor }{ 2 } \, ,
\end{equation}
as observed by Reinhard Zumkeller in 2004 \cite[Sequence A008805]{sloaneonlineseq}.
(It is not clear to us where in the literature this formula first appeared, though specific values
of $p(n,k)$ are well represented in \cite{sloaneonlineseq}, where Sequences A000005, A049820,
A008805, A128508, and A218567--A218573 give the first values of $p(n,k)$ for fixed $k=0, 1, \dots,
10$, and Sequence A097364 paints a general picture of $p(n,t)$.)

We remark that $p(n,2)$ is arithmetically quite different from $p(n,0)$ and $p(n,1)$: namely, $p(n,2)$
is a \Def{quasipolynomial}, i.e., a function that evaluates to a polynomial when $n$ is restricted to
a fixed residue class modulo some (minimal) positive integer, the \Def{period} of the quasipolynomial. (For $p(n,2)$ this period is
2.) Equivalently, the accompanying generating function evaluates to a rational function all of whose poles are rational roots of unity. (See, e.g., \cite[Chapter 4]{stanleyec1} for more on quasipolynomials and their rational generating functions.)
Our goal is to prove closed formulas for these generating functions
\[
  P_t(q) := \sum_{ n \ge 1 } p(n,t) \, q^n .
\]

\begin{theorem}\label{thm:main}
For $t > 1$,
\begin{align*}
  P_t(q)
  &= \frac{ q^{ t-1 } (1-q) }{ (1-q^t) (1-q^{ t-1 }) }
   - \frac{ q^{ t-1 } }{ (1-q^t)^2 (1-q^{ t-1 })^2 (1-q^{ t-2 }) \cdots (1-q^2) } \\
  &\qquad + \frac{ q^t }{ (1-q^t) (1-q^{ t-1 })^2 (1-q^{ t-2 }) \cdots (1-q) } \, .
\end{align*}
\end{theorem}

Written in terms of the usual shorthand $(q)_m := (1-q) (1-q^2) \cdots (1-q^m)$, Theorem \ref{thm:main} says
\[
  P_t(q) = \frac{ q^{ t-1 } (1-q) }{ (1-q^t) (1-q^{ t-1 }) } - \frac{ q^{ t-1 } (1-q) }{ (1-q^t)
(1-q^{ t-1 }) (q)_t } + \frac{ q^t }{ (1-q^{ t-1 }) (q)_t } \, .
\]
Thus $P_t(q)$ is rational for $t>1$, and so $p(n,t)$ is a quasipolynomial in $n$, of degree $t$ and period $\lcm(1, 2, \dots, t)$.
For example, for $t=2$, Theorem \ref{thm:main} gives
\[
  P_2(q)
  = \frac{ q^4 }{ (1-q)^3 (1+q)^2 }
\]
which confirms \eqref{eq:caset2}.
The rational generating function given by Theorem \ref{thm:main} in the case $t=3$ simplifies to
\[
  P_3(q)
  = \frac{ q^5 + q^6 + q^7 - q^8 }{ (1 - q^2)^2 (1 - q^3)^2 } 
\]
which (by way of a computer algebra system or a straightforward binomial expansion) translates to the partition counting function
\begin{align*}
  p(n,3) &= \frac{ 1 }{ 108 } \times \begin{cases}
    n^3 - 18n & \text{ if } n \equiv 0 \bmod 6 , \\
    n^3 - 3n + 2 & \text{ if } n \equiv 1 \bmod 6 , \\
    n^3 - 30n + 52 & \text{ if } n \equiv 2 \bmod 6 , \\
    n^3 + 9n - 54 & \text{ if } n \equiv 3 \bmod 6 , \\
    n^3 - 30n + 56 & \text{ if } n \equiv 4 \bmod 6 , \\
    n^3 - 3n - 2 & \text{ if } n \equiv 5 \bmod 6
  \end{cases} 
\end{align*}
\begin{align*}
  &= \begin{cases}
    m(2m^2 -1) & \text{ if } n = 6m , \\
    m (2m^2 + 1) & \text{ if } n = 6m+1 , \\
    m(2m^2 + 2m -1) & \text{ if } n = 6m+2 , \\
    m(2m^2 + 3m + 2) & \text{ if } n = 6m+3 , \\
    (m-1)(2m^2 -1) & \text{ if } n = 6m-2 , \\
    m^2 (2m-1) & \text{ if } n = 6m-1 .
  \end{cases}
\end{align*}
Using this explicit form of $p(n,3)$, one easily affirms a conjecture about the recursive structure
of $p(n,3)$ given in \cite[Sequence A128508]{sloaneonlineseq} in the positive.

\begin{proof}[Proof of Theorem \ref{thm:main}]
We will use the usual shorthand
\[
  (A)_m := (1-A) (1-A \, q) \cdots (1-A \, q^{ m-1 })
\]
as well as \emph{Heine's transformation} (see, e.g., \cite[p.~38]{andrewstheoryofpartitions})
\begin{equation}\label{eq:heine}
  \sum_{ m \ge 0 } \frac{ (a)_m (b)_m \, z^m }{ (q)_m (c)_m } = \frac{ ( \frac c b )_\infty (bz)_\infty }{ (c)_\infty (z)_\infty } \sum_{ j \ge 0 } \frac{ ( \frac{ abz }{ c } )_j (b)_j ( \frac c b )^j }{
(q)_j (bz)_j }  \, .
\end{equation}
Now we construct the generating function for $p(n,t)$.
A partition of $n$ with difference $t$ between smallest and largest part starts with some part $m$, ends with the part $m+t$, and
could include any of the numbers $m+1, m+2, \dots, m+t-1$ as parts. Translated into geometric series, this gives
\begin{align*}
  P_t(q)
  &= \sum_{ m \ge 1 } \frac{ q^m }{ 1-q^m } \, \frac{ 1 }{ 1-q^{ m+1 } } \cdots \frac{ 1 }{ 1-q^{ m+t-1 } } \, \frac{ q^{ m+t } }{ 1-q^{ m+t } }
   = q^t \sum_{ m \ge 1 } \frac{ q^{ 2m } (q)_{ m-1 } }{ (q)_{ m+t } }
   = q^{ t+2 } \sum_{ m \ge 0 } \frac{ q^{ 2m } (q)_m }{ (q)_{ m+t+1 } } \\
  &= \frac{ q^{ t+2 } }{ (q)_{ t+1 } } \sum_{ m \ge 0 } \frac{ (q)_m (q)_m \, q^{ 2m } }{ (q)_m (q^{ t+2 } )_m }
   \stackrel{ \eqref{eq:heine} }{ = } \frac{ q^{ t+2 } (q^{ t+1 })_\infty (q^3)_\infty }{ (q)_{ t+1 }
(q^{ t+2 })_\infty (q^2)_\infty } \sum_{ j \ge 0 } \frac{ (q^{ -t+2 })_j (q)_j \, q^{ j(t+1) } }{ (q)_j
(q^3)_j } \\
  &= \frac{ q^{ t+2 } }{ (q)_t } \sum_{ j=0 }^{ t-2 } \frac{ (q^{ -t+2 })_j \, q^{ j(t+1) } }{ (q^2)_{
j+1 } }
   = \frac{ q^{ t+2 } }{ (q)_t } \sum_{ j=0 }^{ t-2 } \frac{ (1-q^{ t-2 }) (1-q^{ t-3 }) \cdots (1-q^{
t-j-1 }) (-1)^j q^{ 2j + \binom{ j+1 }{ 2 } } }{ (q^2)_{ j+1 } } \\
  &= \frac{ q^{ t+2 } (1-q) }{ (1-q^t) (1-q^{ t-1 }) } \sum_{ j=0 }^{ t-2 } \frac{ (-1)^j q^{ 2j +
\binom{ j+1 }{ 2 } } }{ (q)_{ j+2 } (q)_{ t-j-2 } }
   = \frac{ q^{ t-1 } (1-q) }{ (1-q^t) (1-q^{ t-1 }) (q)_t } \sum_{ j=0 }^{ t-2 } \qbinom{ t }{ j+2 }
(-1)^j q^{ \binom{ j+3 }{ 2 } } \, .
\end{align*}
Thus, by the $q$-binomial theorem (see, e.g., \cite[p.~36]{andrewstheoryofpartitions})
\begin{align*}
  P_t(q)
  &= \frac{ q^{ t-1 } (1-q) }{ (1-q^t) (1-q^{ t-1 }) (q)_t } \sum_{ j=2 }^{ t } \qbinom{ t }{ j }
(-1)^j q^{ \binom{ j+1 }{ 2 } }
   = \frac{ q^{ t-1 } (1-q) }{ (1-q^t) (1-q^{ t-1 }) (q)_t } \left( (q)_t - 1 + q \qbinom t 1 \right)
\\
  &= \frac{ q^{ t-1 } (1-q) }{ (1-q^t) (1-q^{ t-1 }) } - \frac{ q^{ t-1 } (1-q) }{ (1-q^t) (1-q^{ t-1
}) (q)_t } + \frac{ q^t }{ (1-q^{ t-1 }) (q)_t } \, . \qedhere
\end{align*}
\end{proof}

A natural question concerns the growth behavior of $p(n,t)$. We see in the above example that the quasipolynomial $p(n,3)$ has a
constant leading coefficient, which of course determines the asymptotic growth of $p(n,3)$. Something similar can be said in general.

\begin{corollary}\label{cor:asymptotics}
If $t > 1$ then $p(n,t) = \dfrac{ n^t }{ t \, (t!)^2 } + O(n^{ t-1 })$ as $n \to \infty$.
\end{corollary}

\begin{proof} 
It is well known that the first-order asymptotics of a quasipolynomial stems from the highest-order poles of its rational generating
function. (This follows from first principles, essentially partial-fraction decomposition; see \cite{flajoletsedgewick} for
far-reaching generalizations.)
In our case, $P_t(q)$ has a unique highest-order pole at $q=1$ of order $t$.
Thus the leading coefficient of $p(n,t)$ equals $\frac{ 1 }{ t! }$ times the lowest coefficient of the Laurent series of $P_t(q)$ at
$q=1$ which is
\[
  \lim_{ q \to 1 } \frac{ (1-q)^{ t+1 } (2q^t - q^{ 2t } - q^{ t-1 }) }{ (1-q^t)^2 (1-q^{ t-1 })^2 (1-q^{ t-2 }) \cdots (1-q) }
  = \frac{ 1 }{ t \cdot t! } \, . \qedhere
\]
\end{proof}

Next we shall generalize Theorem \ref{thm:main} by considering \emph{partitions with specified distances}.
Let $p(n, t_1, t_2, \dots, t_k)$ be the number of partitions of $n$ such that, if $\sigma$ is the smallest part then $\sigma + t_1 + t_2 + \dots + t_k$ is the largest part and
each of $\sigma + t_1$, $\sigma + t_1 + t_2$, \dots, $\sigma + t_1 + t_2 + \dots + t_{k-1}$ appear as parts.
We consider the related generating function
\[
  P_{ t_1, \dots, t_k } (q) := \sum_{ n \ge 1 } p(n, t_1, t_2, \dots, t_k) \, q^n .
\]
We note that when $k=1$ this is simply $P_t(q)$ from above.

\begin{theorem}\label{thm:specdist}
For $t := t_1 + t_2 + \dots + t_k > k$,
\[
  P_{ t_1, \dots, t_k } (q) = \frac{ (-1)^k q^{ T - \binom{ k+1 }{ 2 } } \left( \sum_{ j=0 }^{ k } \qbinom t j (-1)^j q^{ \binom{ j+1 }{ 2 } } - (q)_t \right) }{ \qbinom{ t-1} k (1-q^t) (q)_t } \, ,
\]
where $T := k t_1 + (k-1) t_2 + \dots + 2 t_{ k-1 } + t_k$
and $\qbinom A B := \frac{ (q)_A }{ (q)_B (q)_{ A-B } }$. 
\end{theorem}

For example, for $k=2$ and $t_1 = t_2 = 2$, we have $p(11,2,2) = 2$ since $1+1+1+3+5$ and $1+2+3+5$ are the
unique two partitions of 11 that contain three parts whose consecutive distances are 2.
Theorem \ref{thm:specdist} says in this case
\[
  P_{ 2, 2 } (q) = \frac{ q^9 + q^{ 10 } + q^{ 11 } + q ^{ 12 } - q^{ 13 } }{ (1-q^2) (1-q^3)^2 (1-q^4)^2 } 
\]
which translates to
\[
  p(n,2,2) = \frac{ 1 }{ 6912 } \begin{cases}
    3 n^4 - 20 n^3 - 24 n^2 + 288 n & \text{ if } n \equiv 0 \bmod 12 , \\
    3 n^4 - 20 n^3 - 78 n^2 + 492 n - 397 & \text{ if } n \equiv 1 \bmod 12 , \\
    3 n^4 - 20 n^3 - 24 n^2 - 48 n + 304 & \text{ if } n \equiv 2 \bmod 12 , \\
    3 n^4 - 20 n^3 - 78 n^2 + 1260 n - 2781 & \text{ if } n \equiv 3 \bmod 12 , \\
    3 n^4 - 20 n^3 - 24 n^2 - 480 n + 2816 & \text{ if } n \equiv 4 \bmod 12 , \\
    3 n^4 - 20 n^3 - 78 n^2 + 492 n + 155 & \text{ if } n \equiv 5 \bmod 12 , \\
    3 n^4 - 20 n^3 - 24 n^2 + 720 n - 3024 & \text{ if } n \equiv 6 \bmod 12 , \\
    3 n^4 - 20 n^3 - 78 n^2 + 492 n + 35 & \text{ if } n \equiv 7 \bmod 12 , \\
    3 n^4 - 20 n^3 - 24 n^2 - 480 n + 3328 & \text{ if } n \equiv 8 \bmod 12 , \\
    3 n^4 - 20 n^3 - 78 n^2 + 1260 n - 3213 & \text{ if } n \equiv 9 \bmod 12 , \\
    3 n^4 - 20 n^3 - 24 n^2 - 48 n - 208 & \text{ if } n \equiv 10 \bmod 12 , \\
    3 n^4 - 20 n^3 - 78 n^2 + 492 n + 547 & \text{ if } n \equiv 11 \bmod 12 .
  \end{cases}
\]



\begin{proof}[Proof of Theorem \ref{thm:specdist}]
Again we start with the natural generating function
\begin{align*}
  P_{ t_1, \dots, t_k } (q)
  &= \sum_{ m \ge 1 } \frac{ q^m \, q^{ m+t_1 } \, q^{ m+t_1+t_2 } \cdots q^{ m+t_1 + t_2 + \dots + t_k } }{ (1-q^m) (1-q^{ m+1 }) \cdots (1-q^{ m+t_1 + t_2 + \dots + t_k }) }
   = \sum_{ m \ge 1 } \frac{ q^{ (k+1)m + T } }{ (q^m)_{ t+1 } } \\
  &= \sum_{ m \ge 1 } \frac{ q^{ (k+1)m + T } (q)_{ m-1 } }{ (q)_{ m+t } }
   = q^{ T+k+1 } \sum_{ m \ge 0 } \frac{ q^{ (k+1)m } (q)_{ m } }{ (q)_{ m+t+1 } }
   = \frac{ q^{ T+k+1 } }{ (q)_{ t+1 } } \sum_{ m \ge 0 } \frac{ (q)_m (q)_m \, q^{ (k+1)m } }{ (q)_m
(q^{ t+2 })_m } \\
  &\stackrel{ \eqref{eq:heine} }{ = } \frac{ q^{ T+k+1 } (q^{ t+1 })_\infty (q^{ k+2 })_\infty }{
(q)_{ t+1 } (q^{ k+1 })_\infty (q^{ t+2 })_\infty } \sum_{ j \ge 0 } \frac{ (q^{ k+1-t })_j (q)_j q^{
(t+1)j} }{ (q)_j (q^{ k+2 })_j } \\
  &= \frac{ q^{ T+k+1 } (q)_k }{ (q)_t } \sum_{ j=0 }^{ t-k-1 } \frac{ (q^{ -(t-k+1) })_j q^{ (t+1)j}
}{ (q)_{j+k+1} }
\end{align*}
\begin{align*}
  &= \frac{ q^{ T+k+1 } (q)_k }{ (q)_t } \sum_{ j=0 }^{ t-k-1 } \frac{ (1-q^{ t-k-1 }) (1-q^{ t-k-2 })
\cdots (1-q^{ t-k-j }) (-1)^j q^{ \binom j 2 - j(t-k-1) + (t+1)j } }{ (q)_{j+k+1} } \\
  &= \frac{ q^{ T+k+1 } (q)_k }{ (q)_t } \sum_{ j=0 }^{ t-k-1 } \frac{ (q)_{ t-k-1 } (-1)^j q^{ \binom
{j+1} 2 + j(k+1) } }{ (q)_{j+k+1} (q)_{j-k-j-1} } \\
  &= \frac{ q^{ T+k+1 } (q)_k (q)_{ t-k-1 } }{ (q)_t (q)_t } \sum_{ j=0 }^{ t-k-1 } \qbinom t {j+k+1}
(-1)^j q^{ \binom {j+k+2} 2 - \binom {k+2} 2 } \\
  &= \frac{ q^{ T+k+1 } (q)_k }{ \qbinom {t-1} k (1-q^t) (q)_t } \sum_{ j=0 }^{ t-k-1 } \qbinom t
{j+k+1} (-1)^j q^{ \binom {j+k+2} 2 - \binom {k+2} 2 } \\
  &= \frac{ q^{ T - \binom {k+1} 2 } (-1)^{ k+1 } }{ \qbinom {t-1} k (1-q^t) (q)_t } \sum_{ j=k+1 }^{
t } \qbinom t j (-1)^j q^{ \binom {j+1} 2 } \\
  &= \frac{ q^{ T - \binom {k+1} 2 } (-1)^{ k } }{ \qbinom {t-1} k (1-q^t) (q)_t } \left( \sum_{ j=0
}^k \qbinom t j (-1)^j q^{ \binom {j+1} 2 } - (q)_t \right) . \qedhere
\end{align*}
\end{proof}

\bibliographystyle{amsplain}  
\bibliography{bib}  

\end{document}